\documentclass[10pt]{amsart}

\usepackage{hyperref}

\DeclareSymbolFont{rmlargesymbols}{U}{euex}{m}{n}
\DeclareMathSymbol{\rmintop}{\mathop}{rmlargesymbols}{82}
\newcommand{\rmint}{\rmintop\nolimits}

\def\cH{\mathcal{H}}
\def\cO{\mathcal{O}}

\def\al{\alpha}

\def\ga{\gamma}

\def\La{\Lambda}
\def\Si{\Sigma}

\def\Om{\Omega}

\newcommand{\der}{{\rm d}}
\def\P{\mathrm{P}}

\newtheorem{theorem}{Theorem}[section]

\newtheorem{proposition}{Proposition}[section]

\theoremstyle{remark}

\usepackage{amssymb,stmaryrd}
\usepackage{amscd}

\newcommand{\nd}{\nabla}

\author{Matthew Randall}
\address{Department of Mathematics and Statistics\\
Faculty of Science, Masaryk University\\
Kotl\'a\v{r}sk\'a 2, 611 37 Brno\\
Czech Republic}
\email{randallm@math.muni.cz}

\title[Three dimensional near-horizon metrics that are EW]{Three dimensional near-horizon metrics that are Einstein-Weyl}

\subjclass[2010]{53B15, 53B30 (primary), and 83C57 (secondary)} 

\thanks{This work is supported by the Grant agency of the Czech Republic P201/12/G028.}

\begin{document}

\begin{abstract}
We investigate which three dimensional near-horizon metrics $g_{NH}$ admit a compatible 1-form $X$ such that $(X, [g_{NH}])$ defines an Einstein-Weyl structure. We find explicit examples and see that some of the solutions give rise to Einstein-Weyl structures of dispersionless KP type and dispersionless Hirota (aka hyperCR) type.  
\end{abstract}
\maketitle

\section{Introduction}

Let $M^3$ be a three dimensional smooth manifold equipped with a Lorentzian metric. A near-horizon metric on $M^3$ is a Lorentzian metric of the form
\begin{equation}\label{nhg}
g_{NH}=2 \der \nu \left( \der r+r h(x) \der x+\frac{r^2}{2}F(x) \der \nu\right)+\der x \der x,
\end{equation}
where $x$, $\nu$ and $r$ are local coordinates and $h(x)$, $F(x)$ are arbitrary functions of $x$. Near-horizon geometries in higher dimensions are studied in relation to the existence of extremal black holes \cite{dgs}, \cite{KL3}, \cite{lp}, \cite{lrs}, \cite{lsw}.

The near-horizon metric (\ref{nhg}) is derived as follows. For a smooth null hypersurface $\Si^2$ in $M^3$, there exists an adapted coordinate system called Gaussian null coordinates valid in any neighbourhood of $\Si^2$. Imposing that the normal vector field $N^a$ be Killing in $M^3$ implies that $\Si^2$ is a Killing horizon. We further require that the Killing horizon $\Si^2$ be degenerate. This means that the flow of $N^a$ is along affinely parameterised null geodesics on $\Si^2$, or equivalently that the surface gravity of $\Si^2$ is zero. The scaling limit of the degenerate Killing horizon is called a near-horizon geometry, and the metric is of the form (\ref{nhg}).

In dimensions 4 or more, imposing the vacuum Einstein equations on the near-horizon geometry metric $g_{NH}$ give rise to the near-horizon geometry equations on the spatial section of the degenerate Killing horizon. Solutions to this overdetermined system of equations on a compact cross-section are related to the existence of extremal black holes. For further details, see \cite{KL0}, \cite{KL3} and \cite{lp}.

Similarly, we can require that a 3 dimensional near-horizon geometry metric satisfies Einstein's equations. This has been done in \cite{KL3}. In this case, the near-horizon geometry equations reduce to a pair of first order ODEs that can be integrated explicitly. It is also found that global and periodic solutions for $h(x)$ exist on a 1-dimensional cross-section $\cH$ which has the topology of the circle; this is the 1-dimensional analogue of an extremal black hole horizon. 

Instead of asking Einstein's equations to be satisfied, we can ask for a natural generalisation which is to impose the Einstein-Weyl equations instead. This requires an additional structure of a 1-form, to be explained in the next section. Here we investigate whether there exists a 1-form $X$ compatible with (\ref{nhg}) such that $(X,g_{NH})$ satisfies the Einstein-Weyl equations.

In the 37th winter school in geometry and physics, held in Srn\'i, Czech Republic in 2017, plenary lectures on non-expanding horizons including near horizon geometries were presented by Jerzy Lewandowski and lectures on Einstein-Weyl geometry and dispersionless integrable systems were presented by Maciej Dunajski. The latter topics were also covered in the lectures by Evgeny Ferapontov. We hope that attendees of the school would find these topics to be interesting and delve into the rich and steadily expanding field that relates integrable PDEs to general relativity and classical differential geometry.  

\section{Three dimensional Einstein-Weyl geometries}
Einstein-Weyl geometries in 3 dimensions play an important role because of its relationship with the geometry of third order ODEs (see \cite{Nurowski} and \cite{Tod}), twistor theory (see \cite{Hitchin} and \cite{holodisk}) and integrable systems (see \cite{grass}, \cite{ck}, \cite{DFK} and \cite{FerKrug}).
Let $(M^3,[g])$ be a smooth conformal manifold equipped with a conformal class of metrics $[g]$ of (pseudo)-Riemannian signature. In 3 dimensions, this is either Riemannian or Lorentzian. Any 2 representative metric $g$, $\tilde g \in [g]$ are related via
\[
\tilde g= \Om^2 g
\]
for some smooth positive function $\Om$.
A Weyl structure on $(M^3,[g])$ is a torsion-free connection $D$ that preserves the conformal class of metrics. Equivalently, 
\[
D_ag_{bc}=2 X_a g_{bc}
\]
for $g_{ab}$ a representative in the conformal class and $X_a$ a 1-form that is not necessarily closed. 
The Einstein-Weyl equations are the system of equations obtained by requiring that that the symmetric part of the Ricci tensor of the Weyl connection is pure trace, i.e.\
\[
R^{D}_{(ab)}=s g_{ab}
\]
for some function $s$. This system of equations is conformally invariant. Writing $R^D_{(ab)}$ in terms of the Levi-Civita connection $\nd$ for the metric $g_{ab}$ and 1-form $X_a$, we get
\[
\nd_{(a}X_{b)}+X_aX_b+\P_{ab}=\La g_{ab},
\]
where round brackets denote symmetrisation over indices, $\P_{ab}$ is the Schouten tensor of $g_{ab}$ and the trace term 
\[
\La=\frac{1}{3}\left(\nd_aX^b+X_aX^b+\P\right)
\]
is a function (more appropriately, a section of a density valued line bundle and we refer to \cite{conformal} and \cite{ET} for more details).

In \cite{Cartan}, it is shown that a 3-dimensional Lorentzian Einstein-Weyl structure corresponds to a 2 parameter family of totally-geodesic null hypersurfaces in $M^3$. A twistor description is developed in \cite{Hitchin}, where it is shown that real-analytic Lorentzian 3-dimensional Einstein-Weyl structures arise precisely from the Kodaira deformation space of rational normal curves with normal bundle $\cO(2)$ in a 2-dimensional complex manifold.
The 2-form $\der X$ is called the Faraday 2-form of the Weyl structure. If $\der X=0$, then $X$ is locally exact and $g$ can be conformally rescaled to a metric of constant curvature in 3 dimensions. In particular, it implies that $g$ is locally conformally flat. A computation of the Cotton tensor of the near-horizon metric $g_{NH}$ shows that
\begin{proposition}
The metric of the form (\ref{nhg}) is locally conformally flat iff
\begin{equation}\label{cotton}
F'(x)=F(x) h(x). 
\end{equation}
\end{proposition}
\begin{proof}
Condition (\ref{cotton}) implies that the Cotton tensor of $g_{NH}$ is zero, and conversely so.
\end{proof}

\section{Results}
We shall assume that $(M^3,[g_{NH}])$ is smooth.
We consider an ansatz for $X$ of the form
\[
X=c h(x)\der x+X_2 \der r +X_3 \der v,
\]
where $c$ is a constant and the functions $X_2=X_2(x,\nu,r)$, $X_3=X_3(x,\nu,r)$ are to be determined. We also require that $X|_{r=0}=c h(x) \der x$, so that the 1-form $X$ is determined up to a constant multiple $c$ by its value on the spatial section of the degenerate Killing horizon. In the following analysis, it turns out that there is a degenerate case when $c=-\frac{1}{2}$. 
We have the following:
\begin{theorem}\label{nhg-ew-0}
A 3 dimensional near-horizon geometry metric (\ref{nhg}) on $M^3$ of the form 
\begin{align}\label{weierstrass}
g_{NH}=&2 \der \nu \left( \der r+r h(x) \der x+\frac{r^2}{2}e^{{\rmint h(x) dx}}\wp\left(\rmint e^{{\frac{1}{2}\rmint h(x) dx}}dx+a;0,b\right)\der \nu\right)\nonumber\\
&+\der x \der x
\end{align}
where $\wp(z;g_2,g_3)$ is the Weierstrass elliptic function, $a$, $b$ are constants,
and a Weyl connection $X$ of the form
\[
X=-\frac{1}{2}h(x) \der x-2r e^{{\rmint h(x) dx}}\wp\left(\rmint e^{{\frac{1}{2}\rmint h(x) dx}}dx+a;0,b\right) \der \nu
\]
defines an Einstein-Weyl structure $([g_{NH}],X)$ on $M^3$. This depends on $1$ free function of one variable.
\end{theorem}
In particular, taking $h(x)$ to be globally defined and periodic allows $\cH$ to have the topology of a circle.
In the case that $h(x)=0$, the metric simplifies to 
\begin{equation*}
g=2 \der \nu \left( \der r+\frac{r^2}{2}\wp(x+a;0,b)\der \nu\right)+\der x \der x
\end{equation*}
and the 1-form $X$ is given by
\[
X=-2r \wp(x+a;0,b) \der \nu.
\]
We recognise that the function
\[
u(x,\nu,r)=-\frac{r^2}{2}\wp(x+a;0,b)
\]
satisfies the dispersionless Kadomtsev-Petviashvili (dKP) equation given by
\[
2(u_{\nu}-u u_r)_r=u_{xx}
\]
and the Einstein-Weyl structure corresponds to one of dKP type. For further details about such Einstein-Weyl structures we refer to \cite{dKP}.
For generic values of $c$, which also allows for $c=-\frac{1}{2}$, we have
\begin{theorem}\label{nhgewgen}
A 3 dimensional near-horizon geometry metric (\ref{nhg}) on $M^3$ and a Weyl connection $X$ of the form
\[
X=c h(x) \der x+r ((2c+1) h'+c(2c+1) h^2-2 F(x)) \der v
\]
defines an Einstein-Weyl structure if and only if
\[
F(x)=\frac{ h''+4 c h h'+2 c^2 h^3}{2 h}
\]
and $h(x)$ satisfies the $4^{\rm th}$ order ODE
\begin{align}\label{4thode}
&h^3 (h')^2(c-1)^2-\frac{1}{2}(c-1)^2h^4 h''+\frac{9}{4}(c-1)h^2 h' h''\nonumber\\
&-\frac{3}{4}(c-1) h^3 h'''-\frac{1}{2}(h')^2h''+\frac{1}{2}h h' h'''+h(h'')^2-\frac{1}{4}h^2 h''''=0. 
\end{align}
\end{theorem}

\section{Proof}

\begin{proof}[Proof of Theorems \ref{nhg-ew-0} and \ref{nhgewgen}]
We start with an ansatz of the form
\[
X=c h(x)\der x+X_2 \der r +X_3 \der \nu,
\]
where the functions $X_2=X_2(x,\nu,r)$, $X_3=X_3(x,\nu,r)$ are to be determined, and we require that $X|_{r=0}=c h(x) dx$, so that the 1-form $X$ is determined by its value on the spatial section of the degenerate Killing horizon.
We find that substituting this ansatz for $X$ into the Einstein-Weyl equations, the $\der r \der r$ component gives
\[
\partial_r X_2+X_2^2=0, 
\]
which has solutions
\[
X_2=\frac{1}{r+f_1(x,\nu)} \qquad \mbox{or} \qquad X_2=0.
\]
Since the first solution does not restrict to zero on $\{r=0\}$, we take $X_2=0$ instead.
The $\der x \der x$ component gives now 
\[
-2F-\partial_r X_3+c(2c+1) h^2+(2 c+1)h'=0,
\]   
from which we obtain
\[
X_3=(-2F+c(2c+1)h^2+(2c+1)h')r+f_2(x,\nu).
\]
Once again requiring that $X_3|_{r=0}=0$ implies $f_2(x,\nu)=0$. With this, only the $\der x \der \nu$ and $\der \nu \der \nu$ components remain to be solved in the Einstein-Weyl equations.  
The $\der x \der \nu$ component gives
\[
\frac{r}{2}(2c+1)\left(h''-2 h F+4c h h'+2c^2 h^3\right)=0. 
\]
This vanishes identically when $c=-\frac{1}{2}$. Otherwise, we have 
\[
F(x)=\frac{h''+4 c h h'+2 c^2 h^3}{2 h}.
\]
In the first case where $c=-\frac{1}{2}$, the remaining equation in the $\der \nu \der \nu$ component is
\[
-3 F h^2+5 h F'+2 F h'+12 F^2-2 F''=0, 
\]
which has solutions
\[
F(x)=\wp\left(\rmint \exp\left({\frac{1}{2}\rmint h(x) dx}\right) dx+a,0,b\right)\exp\left({\rmint h(x) dx}\right)
\]
whatever $h(x)$ is.
This proves Theorem \ref{nhg-ew-0}.
For the other case, substituting 
\[
F(x)=\frac{h''+4 c h h'+2 c^2 h^3}{2 h}
\] 
gives 
\[
X_3=\frac{c h^3+(1-2 c) h h'-h''}{h}r
\]
and the only remaining $\der \nu \der \nu$ component of the Einstein-Weyl equations gives the $4^{\rm th}$ order ODE
\begin{align*}
&h^3 (h')^2(c-1)^2-\frac{1}{2}(c-1)^2h^4 h''+\frac{9}{4}(c-1)h^2 h' h''\\
&-\frac{3}{4}(c-1) h^3 h'''-\frac{1}{2}(h')^2h''+\frac{1}{2}h h' h'''+h(h'')^2-\frac{1}{4}h^2 h''''=0. 
\end{align*}
Note that this formula is still well-defined for $c=-\frac{1}{2}$.
This proves Theorem \ref{nhgewgen}.
\end{proof}

\section{Explicit solutions}
In this section, we find interesting families of solutions to the 4th order ODE (\ref{4thode}). Consider the following second order nonlinear ODE
\begin{align}\label{2ndh}
h''=\alpha h h'+\beta h^3
\end{align}
with $\al$, $\beta$ constant. We find, upon substituting (\ref{2ndh}) and its derivatives into (\ref{4thode}), that we obtain
\begin{align*}
-\frac{h^3}{4}\big(2(c-1)^2+3\al(c-1)+\al^2-\beta\big)\left(\beta h^4+\al h^2h'-2(h')^2\right)=0.
\end{align*}
Thus solutions to (\ref{2ndh}) automatically satisfy (\ref{4thode}) provided
\begin{equation}\label{abc}
\beta=2(c-1)^2+3\al(c-1)+\al^2.
\end{equation}
Switching independent and dependent variables $(x, h(x)) \mapsto (h, x(h))$, the non-linear second order ODE (\ref{2ndh}) is dual to
\begin{align*}
x''=-\beta h^3 (x')^3-\al h(x')^2,
\end{align*}
which upon setting $y(h)=h^2 x'(h)$, gives an Abel differential equation of the first kind:
\begin{align*}
y'=\frac{1}{h}\left(-\beta y^3-\al y^2+2 y\right).
\end{align*}
This has solutions given by
\begin{align}\label{hy}
h=\frac{\ga\sqrt{y}\exp\left(\frac{\al}{2\sqrt{\al^2+8\beta}}\tanh^{-1}\left(\frac{2\beta y+\al}{\sqrt{\al^2+8\beta}}\right)\right)}{(\beta y^2+\al y-2)^{\frac{1}{4}}},
\end{align} 
and consequently $x$ viewed as a function of $y$ is given by
\begin{align}\label{xy}
x(y)=-\frac{1}{\ga}\rmint \frac{\exp\left(\frac{-\al}{2\sqrt{\al^2+8\beta}}\tanh^{-1}\left(\frac{2\beta y+\al}{\sqrt{\al^2+8\beta}}\right)\right)}{\sqrt{y}(\beta y^2+\al y-2)^{\frac{3}{4}}}\der y.
\end{align}

When $\al=0$, we obtain $\beta=2(c-1)^2$ from (\ref{abc}).
In this case, solutions to the nonlinear ODE (\ref{2ndh})
\[
h''=\beta h^3=2(c-1)^2 h^3
\]
are given by the Jacobi elliptic function
\[
h(x)=a~{\rm sn}\left( a\left(\frac{\sqrt{-2 \beta}}{2}x +b\right), i\right)=a~{\rm sn}\left( a(i(c-1)x +b), i\right),
\]
with $a$, $b$ the constants of integration. Here the elliptic modulus is given by $i=\sqrt{-1}$.
Alternatively, we obtain
\begin{align*}
h=\frac{\ga \sqrt{y}}{(\beta y^2-2)^{\frac{1}{4}}}
\end{align*}
from (\ref{hy}) and
\begin{align*}
x=-\frac{1}{\ga}\rmint\frac{1}{\sqrt{y}(\beta y^2-2)^{\frac{3}{4}}}\der y
\end{align*}
from (\ref{xy}). Passing to $y=\left(\frac{2}{\beta z}\right)^{\frac{1}{2}}$, this gives the expression in terms of hypergeometric functions
\begin{align*}
h=\frac{\ga}{\beta^{\frac{1}{4}}(1-z)^{\frac{1}{4}}} \quad \text{and} \quad
x=\frac{\sqrt{2z}}{2\ga \beta^{\frac{1}{4}}}{}_2F_1\left(\frac{1}{2},\frac{3}{4};\frac{3}{2};z\right).
\end{align*}

When $\beta=0$, we have
\[
(2c-2+\al)(c-1+\al)=0 
\] 
from (\ref{abc}) and therefore $\al=1-c$ or $\al=2-2c$.
For $\al \neq 0$, which implies $c \neq 1$, solutions to (\ref{2ndh}) are given by
\[
h=\frac{1}{\al}\tan\left(\frac{1}{2}\sqrt{2 \ell \al}(x+b)\right)\sqrt{2\ell \al}.
\]
This is periodic but not globally defined. For $\al=0$, or equivalently $c=1$, solutions to (\ref{2ndh}) are given by
\[
h=\ell x+b
\]
with $b$, $\ell$ the constants of integration. The degenerate Killing horizon in this case has the topology of the real line and the metric is not conformally flat. 
If we do not assume either $\al$ or $\beta$ is zero, and consider the case when $c=1$, then equation (\ref{2ndh}) is
\begin{align}\label{habel}
h''=\al h h'+\al^2 h^3.
\end{align}
From (\ref{hy}), we obtain 
\begin{align*}
h=\frac{\ga\sqrt{y}}{(-1)^{\frac{1}{4}}(\al y+2)^{\frac{1}{6}}(1-\al y)^{\frac{1}{3}}}.
\end{align*}
We consider real solutions by taking $\ga=(-1)^{\frac{1}{4}}\varepsilon$.
This gives
\begin{align*}
h^6=\frac{\varepsilon y^3}{(\al y +2) (1-\al y)^2},
\end{align*}
which implies that $y$ is a solution of the cubic equation 
\[
\varepsilon y^3=h^6(\al y+2)(1-\al y)^2.
\]
In the limit $\varepsilon \to 0$, we have $y=-\frac{2}{\al}$ or $y=\frac{1}{\al}$. In these cases we obtain
\[
x(h)=-\rmint \frac{2}{\al h^2}\der h+b=\frac{2}{\al h}+b \mbox{~and~} x(h)=\rmint \frac{1}{\al h^2}\der h+b=-\frac{1}{\al h}+b.
\] 
Hence 
\[
h(x)=\frac{2}{\al(x-b)} \mbox{~and~} h(x)=-\frac{1}{\al(x-b)}
\]
satisfy (\ref{4thode}) with the parameter $c=1$. For these solutions $h(x)$ is singular on the degenerate Killing horizon. These solutions give conformally flat metrics when $\al$ is chosen so that $h(x)=\frac{-2}{x-b}$ or $\frac{1}{x-b}$.
In fact when $c=1$, (\ref{4thode}) can be integrated to give
\begin{align}\label{3rdode1}
-\frac{1}{4}h^2 h'''+h h' h''-\frac{1}{2}(h')^3=0.
\end{align}
Passing to $h(x)={\rm e}^{f(x)}$, we see that (\ref{3rdode1}) is satisfied iff $f(x)$ satisfies the non-linear $3^{\rm rd}$ order ODE
\begin{equation}\label{nlode}
f'''-f' f''-(f')^3=0. 
\end{equation}
This ODE is none other than (\ref{habel}) with $f'$ replacing $h$ and $\al=1$.
This gives the following additional solutions to (\ref{4thode})
\[
h(x)=(x-b)^2 \mbox{~and~} h(x)=\frac{1}{x-b}.
\]
The conformal structure for the first $h(x)$ is not flat and $h(x)$ is globally defined on the line horizon, while the conformal structure for the second is flat, which can be rescaled to a metric of constant scalar curvature.  

Interestingly, we can relate the solutions to (\ref{2ndh}) with $c=-1$, $\al=2$, $\beta=0$ to dispersionless Hirota type or hyperCR Einstein-Weyl structures.
For a review of such structures see \cite{hyperCR}, \cite{hyperCR2}. 
Recall that a metric of the form
\begin{align*}
g=&(\der x+H_r d\nu)^2-4\left(\der r-H_x \der v\right) \der \nu\\
=&2\der \nu\left(-2\der r+H_{r}\der x+(2H_x+\frac{H_r^2}{2})\der \nu\right)+\der x \der x 
\end{align*}
where $H=H(x,\nu,r)$
and a Weyl connection of the form
\[
X=\frac{1}{2}H_{rr}\der x+\frac{1}{2}(H_{r}H_{rr}+2H_{xr}) \der \nu
\]
defines a hyperCR Einstein-Weyl structure iff the hyperCR equation 
\begin{align}\label{hypercr}
H_{x}H_{rr}-H_rH_{xr}-H_{xx}+H_{r\nu}=0
\end{align}
is satisfied. 
A particular family of solutions to this equation can be found:
\[
H(x,\nu,r)=j\tanh^3\left(\frac{a^2}{b}r+b \nu+ax+e\right)+k \tanh\left(\frac{a^2}{b}r+b \nu+ax+e\right)+l
\]
is a $6$ parameter family of solutions satisfying (\ref{hypercr}) depending on constants $a$, $b$, $e$, $j$, $k$, $l$.
Aligning the 1-form in the hyperCR case with our ansatz for $X$, we require that
\[
X|_{r=0}=\frac{1}{2}H_{rr}\der x+\frac{1}{2}(H_{r}H_{rr}+2H_{xr}) \der \nu|_{r=0}=c h(x)\der x.
\]
Hence
\[
H=c h(x) r^2+f_2(x,\nu)
\]
for some function $h(x)$ and $f_2(x,\nu)$. Additionally, to align the hyperCR metric with the near-horizon metric (\ref{nhg}), we require $f_2(x,\nu)=0$. 
Plugging this solution for $H$ back into the metric and 1-form $X$, we discover that the Einstein-Weyl equations are satisfied iff
\begin{align*}
h''=&-2 c h h'.
\end{align*}
This has solutions given by
\[
h(x)=\frac{\sqrt{c \ell}}{c}\tanh\left(\sqrt{c \ell}(x+b)\right)
\]
for $c \neq 0$. 
Redefining coordinates $r \mapsto \tilde r=-2r$ and renaming $r$, 
we have
\begin{proposition}\label{hypercr1}
A 3 dimensional Lorentzian metric on $M^3$ of the form
\begin{equation}\label{tan}
g=2 \der \nu \left( \der r-c h r \der x+\frac{r^2}{2}(c h'+c^2 h^2)\der \nu\right)+\der x \der x
\end{equation}
where $c \neq 0$ with
\[
h(x)=\frac{\sqrt{c \ell}}{c}\tanh\left(\sqrt{c \ell}(x+b)\right)
\]
satisfying the ODE
\[
h''=-2 c h h'
\]
and a Weyl connection $X$ of the form
\begin{align*}
X=&c h \der x- c r (c h^2+h')\der \nu
\end{align*}
defines a hyperCR Einstein-Weyl structure on $M^3$. 
Comparing the metric of the form (\ref{tan}) to the near horizon metric (\ref{nhg}), we see that setting $c=-1$ gives the near-horizon metric
\begin{equation*}
g_{NH}=2 \der \nu \left( \der r+ h r \der x+\frac{r^2}{2}(- h'+h^2)\der \nu\right)+\der x \der x
\end{equation*}
and Weyl connection
\begin{align*}
X=&-h \der x+ r (-h^2+h')\der \nu.
\end{align*}
 
In particular, for $c=-1$, the ODE $h''=2 h h'$ that $h(x)$ satisfies
agrees with the solution (\ref{2ndh}) to (\ref{4thode}) with the parameters $\beta=0$, $c=-1$, $\al=1-c=2$.
\end{proposition}

\section{Further remarks and outlook}
It is curious that solving for the Einstein-Weyl equations for a near-horizon metric (\ref{nhg}) give rise to dKP Einstein-Weyl structures for $c=-\frac{1}{2}$ as in Theorem \ref{nhg-ew-0} and hyperCR Einstein-Weyl structures for $c= -1$ as in Theorem \ref{nhgewgen} and Proposition \ref{hypercr1}. Do other parameters of $c$ give rise to other interesting EW structures? 
The metric in Gaussian null coordinates on $M^3$ with a smooth Killing horizon (see \cite{KL3}) is given by
\begin{equation}\label{dk}
g=2 \der \nu \left( \der r+r h(x,r) \der x+\frac{r}{2}F(x,r) \der \nu\right)+\gamma(x,r)^2 \der x \der x.
\end{equation}
In \cite{descendants}, the authors investigated the case where (\ref{dk}) is Einstein with non-zero cosmological constant. 
Similarly, we can investigate whether the metric of the form (\ref{dk}) admits a compatible 1-form $X_a$ such that $([g],X)$ is Einstein-Weyl.
Some computations have been made but the calculations are considerably more involved than those presented here.

\end{document}